\documentclass[11pt]{smfart}

\usepackage[all]{xy}
\usepackage{amscd,amssymb,amsfonts,amsmath}
\usepackage[francais]{babel}
\usepackage{smfthm}
\usepackage{graphics}
\usepackage{epsfig}

\newcommand\mypagesizel{
\textwidth= 6.2in
\textheight=9in
\voffset-.65in
\hoffset -0.47in
\marginparwidth=56pt
}

\newcommand{\Pic}{\textup{Pic}}
\newcommand{\Spec}{\textup{Spec}}

\newcommand{\NS}{\textup{NS}}

\renewcommand{\H}{\textup{H}}

\newcommand{\cad}{c'est-\`a-dire}

\renewcommand{\phi}{\varphi}

\newcommand{\lra}{\longrightarrow}

\newcommand{\Fabs}{\textit{F}_{\textup{abs}}}
\newcommand{\Fabsbar}{{\bar F}_{\textup{abs}}}
\newcommand{\Fabsv}{\textit{F}_{\textup{abs,\textit{v}}}}
\newcommand{\Fcris}{\textit{F}_{\textup{cris}}}
\newcommand{\Fcrisbar}{{\bar F}_{\textup{cris}}}

\renewcommand{\le}{\leqslant}
\renewcommand{\ge}{\geqslant}

\newcommand{\bC}{\textup{\textbf{C}}}

\newcommand{\bF}{\textbf{F}}

\newcommand{\bQ}{\textup{\textbf{Q}}}
\newcommand{\bR}{\textup{\textbf{R}}}

\newcommand{\bZ}{\textup{\textbf{Z}}}

\newcommand{\cF}{\mathcal{F}}

\newcommand{\cH}{\mathcal{H}}

\newcommand{\cO}{\mathcal{O}}

\newcommand{\cS}{\mathcal{S}}

\newcommand{\cV}{\mathcal{V}}

\newcommand{\cX}{\mathcal{X}}

\newcommand{\Id}{\textup{Id}}
\newcommand{\Gal}{\textup{Gal}}
\newcommand{\Gl}{\textup{GL}}
\newcommand{\rang}{\textup{rang}}
\newcommand{\Tr}{\textup{Tr}}
\newcommand{\End}{\textup{End}}

\newcommand{\Pet}{\textup{P}_{\textup{\'et}}}
\newcommand{\Ret}{\textup{R}_{\textup{\'et}}}
\newcommand{\Het}{\textup{H}_{\textup{\'et}}}

\newcommand{\Pb}{\textup{P}_{\textup{B}}}
\newcommand{\Rb}{\textup{R}_{\textup{B}}}
\newcommand{\Hb}{\textup{H}_{\textup{B}}}

\newcommand{\Hcris}{\textup{H}_{\textup{cris}}}

\newcommand{\Hdr}{\textup{H}_{\textup{dR}}}

\renewcommand{\ker}{\textup{Ker}}

\mypagesizel

\begin{document}

\title[R\'eductions des vari\'et\'es symplectiques irr\'eductibles]{Invariants de Hasse-Witt des r\'eductions de
certaines vari\'et\'es symplectiques irr\'eductibles}

\author{St\'ephane \textsc{Druel}}

\address{Institut Fourier\\ UMR 5582
du CNRS\\ Universit\'e Grenoble 1, BP 74\\ 38402 Saint Martin d'H\`eres,
France.}

\email{druel@ujf-grenoble.fr}

\urladdr{http://www-fourier.ujf-grenoble.fr/~druel/}

\maketitle

\section{Introduction}

On d\'efinit \textit{l'invariant de Hasse-Witt} d'une vari\'et\'e alg\'ebrique $X$ projective et lisse sur 
un corps parfait de caract\'eristique
$>0$
comme le rang stable du Frobenius absolu agissant
sur $\H^{\dim(X)}(X,\cO_X)$.

Soit $X$ une vari\'et\'e d\'efinie sur $\bZ$ avec $\det(\Omega_X^1)\simeq\cO_X$.
On s'int\'eresse ici aux invariants de Hasse-Witt des vari\'et\'es obtenues par r\'eduction de
$X$ modulo un nombre premier.

On sait, d'apr\`es le th\'eor\`eme de d\'ecomposition de Bogomolov (voir \cite[Th\'eor\`eme 1]{beauville83}),
que toute vari\'et\'e (lisse) complexe compacte k\"ahl\'erienne
avec $c_1(X)=0\in\Hdr(X,\bC)$
est \`a un rev\^etement \'etale fini et \`a un isomorphisme pr\`es un
produit 
$$T\times \prod_{i\in I} Y_i\times\prod_{j\in J} Z_j$$
o\`u 
\begin{enumerate}
\item[$\bullet$]$T$ est un tore complexe,
\item[$\bullet$]$Y_i$ est une vari\'et\'e de \textit{Calabi-Yau}, c'est-\`a-dire, $Y_i$
est projective, simplement connexe (de dimension $\ge 3$) et
$\H^0(Y_i,\Omega_{Y_i}^{\bullet})=\bC\oplus\bC\omega_i$ o\`u
$\omega_i$ est
une forme diff\'erentielle  de degr\'e $\dim(Y_i)$
partout non nulle,
\item[$\bullet$]$Z_j$ est \textit{symplectique irr\'eductible},
c'est-\`a-dire, $Z_j$
est simplement connexe et $\H^0(Z_j,\Omega_{Z_j}^{\bullet})=\bC[\Omega_j]$ o\`u 
$\Omega_j$ est
une $2$-forme partout non d\'eg\'en\'er\'ee.
\end{enumerate}

On a tr\`es peu d'exemples de vari\'et\'es symplectiques irr\'eductibles. Beauville a constuit deux familles de
vari\'et\'es symplectiques (irr\'eductibles) de dimension $2m$ (pour tout entier $m \ge 1$)~: l'une est le sch\'ema de
Hilbert \og des points\fg~$S^{[m]}$ d'une surface $K3$ $S$ (\cite[Th\'eor\`eme 3]{beauville83}) et l'autre est la
varit\'et\'e de Kummer g\'en\'eralis\'ee $K_m(A)$ d'une surface ab\'elienne $A$ qui par d\'efinition est la fibre
au-dessus de $0$ du morphisme d'Albanese de $A^{[m+1]}$ (\cite[Th\'eor\`eme 4]{beauville83}). O'Grady a construit
deux autres familles d'exemples~: l'une une famille de vari\'et\'es de dimension 10 (\cite{ogrady99}) et deuxi\`eme
nombre de Betti \'egal \`a 24 (\cite{rapagnetta08})
et l'autre une famille de vari\'et\'es de dimension 6 et deuxi\`eme nombre de Betti \'egal \`a 8 (\cite{ogrady03}).

On sait peu de choses des invariants de Hasse-Witt des r\'eductions modulo un
nombre premier des vari\'et\'es ab\'eliennes et des vari\'et\'es de Calabi-Yau.
On sait toutefois, qu'\'etant donn\'e une courbe elliptique $E$ sur $\bQ$, 
ses r\'eductions modulo un nombre premier ont un invariant de Hasse-Witt nul pour une infinit\'e de nombres premiers
(\cite{elkies87}) et
on montre facilement que la surface de Kummer associ\'ee \`a $E\times E$ a la m\^eme propri\'et\'e. 

\medskip

On obtient
le r\'esultat suivant.

\begin{theo}\label{theo:principal}
Soient $K$ un corps de nombres et $\cO_K$ l'anneau des entiers de $K$. On note $\bar K$ une cl\^oture alg\'ebrique
de $K$ et on fixe un nombre premier $\ell$.
Soit $X$ une vari\'et\'e symplectique irr\'eductible polarisable sur $K$. Soient
$f : \cX \to \Spec(\cO_K)$ un mod\`ele entier de $X$ et
$\cV$ un ensemble fini de places non archim\'ediennes de $K$ tel que $f$ soit lisse au-dessus de l'ouvert
$\Spec(\cO_K[\cV^{-1}])$ de $\Spec(\cO_K)$.
On suppose
ou bien $\textup{rang}(\NS(X\otimes\bar K)) \ge 2$
ou bien $\dim_{\bQ_\ell}(\Het^2(X\otimes\bar K,\bQ_\ell))$ pair.
Il existe alors un ensemble $\Sigma$ de places finies de $K$ de densit\'e $>0$
tel que pour tout
$v\in\Sigma\setminus\cV$ l'invariant de Hasse-Witt de $\cX_v$ soit non nul. On peut supposer la densit\'e de
$\Sigma$  \'egale
\`a 1 quitte \`a remplacer $K$ par une extension finie $L$ de $K$
convenable.
\end{theo}

Ogus a d\'emontr\'e un r\'esultat analogue pour une surface ab\'elienne quelconque (\cite[Expos\'e
VI]{deligneLN900}).
Joshi et Rajan (\cite{joshi01}) puis Bogomolov et Zahrin (\cite{bogomolovzahrin}) ont d\'emontr\'e l'\'enonc\'e
ci-dessus lorsque $\dim(X)=2$, autrement dit lorsque $X$ est une surface $K3$, auquel cas $b_2(X)=22$. On utilise ici
les m\^emes outils. 

Soient $X$ une vari\'et\'e symplectique irr\'eductible 
d\'efinie sur un corps premier $\bF_p$ de caract\'eristique $p$ et 
$\Omega\in\textup{H}^0(X,\Omega_X^2)\setminus \{0\}$ ($\Omega$ est par d\'efinition ferm\'ee).
On montre facilement que l'invariant de Hasse-Witt
de $X$ est non nul si et seulement si
$C(\Omega)\neq 0$ o\`u $C$ d\'esigne l'op\'eration de Cartier. 
On fixe un nombre premier $\ell\neq p$.
On montre ensuite, \`a peu de choses pr\`es, que si la valuation $p$-adique 
de l'une des valeurs propres du \og Frobenius g\'eom\'etrique\fg~agissant sur
$\Het^2(X\otimes\bar \bF_p,\bQ_\ell)$ est nulle alors $C(\Omega)\neq 0$ et l'invariant  de Hasse-Witt de
$X$ est donc non nul. On conclut avec un
\'enonc\'e du type th\'eor\`eme de densit\'e de \v{C}ebotarev d\^u \`a Serre.

\medskip

On d\'emontre au passage le r\'esultat suivant (voir \'egalement \cite{andre96}[Theorem 1.6.1] pour un r\'esultat
analogue dans le cas \og global\fg).

\begin{theo}
Soit $X$ une vari\'et\'e symplectique irr\'eductible polarisable d\'efinie sur un corps de nombres $K$.
On note $\bar K$ une
cl\^oture alg\'ebrique
de $K$. On fixe un nombre premier $\ell$ et on consid\`ere la repr\'esentation 
$\rho : \Gal(\bar K/K) \lra \Gl(\Het^2(X\otimes \bar K,\bQ_\ell))$.
On suppose ou bien $\textup{rang}(\NS(X\otimes\bar K)) \ge 2$
ou bien $\dim_{\bQ_\ell}(\Het^2(X\otimes\bar K,\bQ_\ell))$ pair.
Il existe alors un ensemble fini $\cV$ de places ultram\'etriques de $K$ tel que, si $v$ est une place finie de
$K$ et $v\not\in \cV$ alors $\rho$ est non ramifi\'ee en $v$ et l'\'el\'ement de Frobenius en $v$
$F_{v,\rho}\in\Gl(\Het^2(X\otimes \bar K,\bQ_\ell))$ est semi-simple.
\end{theo}

\begin{enonce*}[remark]{Remerciements}Je remercie D. Huybrechts de m'avoir indiqu\'e que \cite{andre96} contient une
d\'emonstration de l'\'enonc\'e de la conjecture de Tate pour les vari\'et\'es simplectiques irr\'eductibles sur les
corps de nombres sous la seule hypoth\`ese $b_2 \ge 4$ alors que dans une pr\'ec\'edente version de ce texte j'en
donnais une d\'emonstration (incompl\`ete par ailleurs) dans le cas particulier (auquel Andr\'e se ram\`eme) o\`u le
rang du groupe de
N\'eron-Severi est $\ge 2$.
\end{enonce*}

\section{Notations et rappels}

\subsection{}On se donne un corps $k$ et on fixe une cl\^oture alg\'ebrique $\bar k$ de $k$. 
On appelle \textit{vari\'et\'e (alg\'ebrique) sur $k$} ou \textit{$k$-vari\'et\'e (alg\'ebrique)} un sch\'ema de type
fini sur $k$. Soit $X$ une $k$-vari\'et\'e.
On notera $X\otimes\bar k$ le produit de
$X$ et $\Spec(\bar k)$ au-dessus de $\Spec(k)$.
\begin{defi}On dit qu'une $k$-vari\'et\'e $X$ lisse 
est \textit{symplectique} 
si $X$ est g\'eom\'etriquement int\`egre et 
s'il existe une $2$-forme ferm\'ee
$\bar\Omega\in\H^0(X\otimes \bar k,\Omega_{X\otimes \bar k}^2)$ qui est non d\'eg\'en\'er\'ee en tout point.
On dit qu'une vari\'et\'e symplectique $X$ propre sur $k$ est symplectique irr\'eductible si 
$\dim_{\bar k}(\H^1(X\otimes \bar k,\cO_{X\otimes \bar k}))=0$ et
$\dim_{\bar k}(\H^0(X\otimes \bar k,\Omega_{X\otimes \bar k}^2))=1$ ou, ce qui revient au m\^eme,
si $\dim_{k}(\H^1(X,\cO_X))=0$ et
$\dim_{k}(\H^0(X,\Omega_X^2))=1$.
\end{defi}

\begin{rema}Soit $X$ une vari\'et\'e propre et lisse sur $k$. On suppose $X$ symplectique (en particulier g\'eom\'etriquement int\`egre).
On sait bien qu'on a 
$$\H^0(X\otimes \bar k,\Omega_{X\otimes\bar k}^2)\simeq \H^0(X,\Omega_{X}^2)\otimes_k\bar k$$
et
$$\H^0(X\otimes \bar k,Z\Omega_{X\otimes\bar k}^2)\simeq \H^0(X,Z\Omega_{X}^2)\otimes_k\bar k$$
o\`u l'on a pos\'e $Z\Omega_X^2=\ker(d_X : \Omega_X^2 \to \Omega_X^{3})$.
On en d\'eduit que l'application polynomiale 
$\displaystyle{\H^0(X,\Omega_{X}^2)\ni\Omega \mapsto \Omega^{\wedge m}
\in \H^0(X,\cO_X)\simeq k}$ o\`u $2m=\dim(X)$ n'est pas identiquement nulle puisqu'elle ne l'est pas apr\`es extension
des scalaires \`a $\bar k$ et que sa rectriction \`a 
$\H^0(X,Z\Omega_{X}^2)$ ne l'est pas non plus.
Il existe donc une $2$-forme ferm\'ee 
$\Omega\in\H^0(X,\Omega_{X}^2)$ non
d\'eg\'en\'er\'ee en tout point.
\end{rema}

\subsection{Th\'eories cohomologiques}On se donne un corps $k$ 
de caract\'eristique $p \ge 0$,
une vari\'et\'e $X$ propre sur $k$ et
on fixe une
cl\^oture alg\'ebrique $\bar k$ de $k$.

On suppose $k=\bC$ et on se donne un anneau de coefficients $\Lambda$. On note alors $\Hb^{\bullet}(X,\Lambda)$ la
cohomologie singuli\`ere de l'espace analytique complexe $X(\bC)$ \`a coefficients dans $\Lambda$.

On se donne un entier premier $\ell\neq p$. On note 
$\Het^{\bullet}(X,\bZ_\ell):=\underset{\longleftarrow}{\lim}\,\Het^{\bullet}(X,\bZ/\ell^n\bZ)$ et 
$\Het^{\bullet}(X,\bQ_\ell):=\Het^{\bullet}(X,\bZ_\ell)\otimes_{\bZ_{\ell}}\bQ_{\ell}$ les groupes de cohomologie
$\ell$-adique \`a valeurs dans $\bZ_\ell$ et $\bQ_\ell$ respectivement.

On suppose $k$ parfait et $p>0$. On note $W=W(k)$ l'anneau des vecteurs de Witt de
$k$, $K_0$ son corps des fractions et $\Hcris^{\bullet}(X/W)$ la cohomologie cristalline de $X$. 

On sait que lorsque $X$ est projective et lisse sur $k$ (comme ci-dessus), $\Hb^{\bullet}(X,\bQ)$,
$\Het^{\bullet}(X\otimes\bar k,\bQ_\ell)$ et
$\Hcris^{\bullet}(X/W)\otimes_W K_0$ sont des alg\`ebres anti-commutatives gradu\'ees de dimensions finies sur  
$\bQ$, $\bQ_{\ell}$ et $K_0$ respectivement et qu'on a une formule de K\"unneth, une dualit\'e de Poincar\'e, des
th\'eor\`emes de Lefschetz faible et fort et enfin des applications cycle.

On a \'egalement des structures suppl\'ementaires sur chacun de ces groupes de cohomologie. 
On suppose $k$ fini de cardinal $q=p^a$ ($a>0$).
On note $\Fabs : X \to X$ le Frobenius absolu. Il agit sur par fonctorialit\'e sur $\Hcris^{\bullet}(X/W)$~; 
$\Fabs^a$ agit lin\'eairement sur $\Hcris^{\bullet}(X/W)$ et $\Hcris^{\bullet}(X/W)\otimes_W K_0$.
On notera $K_0(r)$ pour $r\in\bZ$ le $K_0$ espace vectoriel de dimension $1$ sur $K_0$ sur lequel $\Fabs^a$ agit par
multiplication par $q^r$ et pour tout $K_0[\Fabs^a]$-module $V$, on note
$V(r):=V\otimes_{K_0} K_0(r)$.

On suppose $k$ quelconque (et $\ell\neq p$)~; le groupe de Galois absolu $\Gal(\bar k/k)$ agit sur 
$\Het^{\bullet}(X \otimes \bar k,\bZ_\ell)$
et $\Het^{\bullet}(X\otimes \bar k,\bQ_\ell)$ par transport de structures. On consid\`ere les $\Gal(\bar k/k)$-modules
$\bZ/\ell^n\bZ(1):=\mu_{\ell^n}(\bar k)$,
$\bZ_\ell(1):=\underset{\longleftarrow}{\lim}\,\bZ/\ell^n\bZ(1)$ et
$\bQ_{\ell}(1):=\bZ_\ell(1)\otimes_{\bZ_\ell} \bQ_\ell$ (le module de Tate $\ell$-adique).
On note aussi, pour $r \ge 1$, 
$\bZ_{\ell}(r):=\bZ_{\ell}(1)^{\otimes r}$ (resp.
$\bQ_{\ell}(r):=\bQ_{\ell}(1)^{\otimes r}$), et enfin $V(r):=V\otimes_{\bZ_{\ell}} \bZ_{\ell}(r)$
(resp. $V(r):=V\otimes_{\bQ_{\ell}} \bQ_{\ell}(r)$) pour tout $\bZ_\ell[\Gal(\bar k/k)]$-module 
(resp. $\bQ_\ell[\Gal(\bar k/k)]$-module) $V$.

On suppose $k=\bC$. On rappelle que les groupes $\Hb^{m}(X,\bZ)$ ($m\ge 0$) portent une structure de Hodge
enti\`ere pure de poids $m$. On note, pour $r \ge 0$,  
$\bZ(r)$ la structure de hodge enti\`ere de poids $2r$, de type ${(-r,-r)}$ avec $\bZ(r)_{\bC}=\bC$
et $\bZ(r)_{\bZ}=2i\pi\bZ\subset \bC$. On note aussi $\Hb^{m}(X,\bC)(r)$ la structure de Hodge
$\Hb^{m}(X,\bC)\otimes_\bZ \bZ(r)$ correspondante. 
On note enfin $\Hb^{m}(X,\bQ)(r):=\Hb^{m}(X,(2i\pi)^r\bQ)\subset \Hb^{m}(X,\bC)$ ($r\ge 0$).
On rappelle aussi que, par le th\'eor\`eme de comparaison
(\cite[Expos\'e XI]{sga4}), on a un isomorphisme $\Hb^{m}(X,\bQ_\ell)\simeq \Het(X,\bQ_\ell)$, qui induit, via
l'exponentielle, des isomorphismes $\Hb^{m}(X,\bQ_\ell)(r)\simeq \Het(X,\bQ_\ell)(r)$ (voir
\cite[I.1]{deligneLN900}).

On suppose $k$ quelconque. On note $\Hdr(X,k)$ la cohomologie de de Rham de $X$ sur $k$, c'est-\`a-dire
l'hypercohomologie (pour la topologie de Zariski) du complexe de de Rham $\Omega_X^{\bullet}$.

On note enfin $\H^{\bullet}(X,\cF)$ la cohomolgie du faisceau de groupes ab\'eliens $\cF$ pour la topologie de Zariski.

\subsection{Invariant de Hasse-Witt}Soient $k$ un corps parfait de caract\'eristique $p>0$ et 
$\sigma$ l'automorphisme de Frobenius de $k$. Soient
$V$ un espace vectoriel de dimension finie sur $k$ et $\varphi$ un endomorphsime $\sigma$-lin\'eaire de $V$. On a une
d\'ecomposition en somme directe 
$V=V_{ss}\oplus V_{nil}$ o\`u $V_{ss}=\cap_{i\ge 0} \varphi^i(V)$ et $V_{nil}=\cup_{i\ge 0}\ker(\varphi^i)$~; 
les espaces vectoriels $V_{ss}$ et $V_{nil}$ sont $F$-stables,
la
restriction de $F$ \`a $V_{ss}$ est bijective et sa restriction \`a $V_{nil}$ est nilpotente. On appelle \textit{rang
stable} de $\varphi$ la dimension de $V_{ss}$ sur $k$. On remarque que le rang stable de $V$ sur $k$ est \'egal au rang
stable de $V\otimes\bar k$ sur $\bar k$.

Soit $X$ une vari\'et\'e alg\'ebrique de dimension $d$ propre sur $k$. On note $\Fabs : X \to X$ le
morphisme de Frobenius absolu. On appelle \textit{invariant de Hasse-Witt} de $X$ (sur $k$) le rang stable de $\Fabs$
agissant sur
$\H^d(X,\cO_X)$. On note $\bar\sigma$ l'automorphisme de Frobenius de $\bar k$ et 
$\Fabsbar : X\otimes\bar k \to X\otimes k$ le Frobenius absolu.
On a bien s\^ur un isomophisme $\H^d(X\otimes\bar k,\cO_{X\otimes\bar k})\simeq\H^d(X,\cO_X)\otimes_k\bar k$ tel que
$\Fabsbar=\Fabs\otimes \bar k$, et l'invariant de Hasse-Witt de $X$ (sur $k$) est donc \'egal \`a celui de
$X\otimes\bar k$ (sur $\bar k)$.

\section{Quelques pr\'eliminaires}

On consid\`ere dans tout ce paragraphe un corps parfait $k$ de caract\'eristique $p>0$.

\subsection{Op\'eration de Cartier}
Soit $X$ une vari\'et\'e alg\'ebrique lisse sur $k$ de
dimension $\dim(X)=d$.
On note $\sigma$ l'automorphisme de Frobenius de $k$ et $\Fabs : X \to X$ le morphisme de Frobenius absolu.
On rappelle
qu'il existe un unique isomorphisme $\sigma$-lin\'eaire d'alg\`ebres gradu\'ees 
$$C^{-1}:\bigoplus \Omega_X^i \lra \bigoplus \cH^i(\Omega_X^{\bullet})$$
tel que $C^{-1}(s)=s^p$ si $s\in\cO_X$ et $C^{-1}(d_Xs)$ soit la classe de $s^{p-1}d_Xs$ ($\cH^i(\Omega_X^{\bullet})$
d\'esigne la cohomologie en degr\'e $i$ du complexe $\Omega_X^{\bullet}$)~; $C$ est appel\'e l'op\'erateur de
Cartier. 

On note $B\Omega_X^i=d_X\Omega_X^{i-1}$ et $Z\Omega_X^i=\ker(d_X : \Omega_X^i \to \Omega_X^{i+1})$ pour tout
$i\in\{0,\ldots,\dim(X)\}$~; ce sont des faisceaux de groupes ab\'eliens. 
On note encore $C$ l'application induite par l'op\'erateur de Cartier
$$Z\Omega_X^i \to Z\Omega_X^i/B\Omega_X^i=\cH^i(\Omega_X^{\bullet}) \to \Omega_X^i.$$

On rappelle enfin
que l'application $k$-lin\'eaire
$${\Fabs}_*Z\Omega_X^{d}={\Fabs}_*\omega_X={\Fabs}_*\Fabs^{!}\omega_X\to \omega_X$$
est une \og application trace\fg~dans la dualit\'e relative au morphisme (fini) $\Fabs$.

\begin{lemm}\label{lemm:hassewittversusscindage}
Soit  $X$ une vari\'et\'e alg\'ebrique de dimension $d=2m \ge 2$ propre et lisse sur 
$k$.
On suppose 
qu'il existe une $2$-forme ferm\'ee $\Omega$ telle qu'on ait $\textup{H}^0(X,\Omega_X^2)=k\cdot\Omega$ et 
$\Omega^{\wedge m}\neq 0$.
\begin{enumerate}
\item Si $C(\Omega)\neq 0$ alors
l'invariant de
Hasse-Witt de $X$ est non nul.
\item On suppose de plus  $\H^0(X,\omega_X)=
k\cdot\Omega^{\wedge m}$. 
On a alors $C(\Omega)\neq 0$ si et seulement si l'invariant de Hasse-Witt de $X$ est non nul.
\end{enumerate}
\end{lemm}
\begin{proof}
On sait que l'application
$$ \H^0(X,\omega_X)\simeq\H^0(X,{\Fabs}_*\Fabs^{!}\omega_X)\to \H^0(X,\omega_X)$$
induite par \og l'application trace\fg~est
duale
de l'application de Hasse-Witt
$$F_{\textup{abs}}^* : \H^{d}(X,\cO_X) \lra \H^{d}(X,\cO_X).$$
On en d\'eduit en particulier que l'invariant de Hasse-Witt de $X$ est non nul si et seulement s'il existe
une forme $\omega$ sur $X$ de degr\'e $d$
telle que $C(\omega)\neq 0$.

On a $C(\Omega)=\lambda\Omega$, pour un scalaire $\lambda\in k$ convenable. On a donc
$$C(\underbrace{\Omega\wedge\cdots \wedge\Omega}_{\textup{m facteurs}})=\underbrace{C(\Omega)\wedge\cdots\wedge
C(\Omega)}_{\textup{m facteurs}}=\lambda^m \Omega.$$
On conclut maintenant facilement.
\end{proof}

On montre avec des arguments analogues le r\'esultat suivant.

\begin{lemm}
Soit  $X$ une vari\'et\'e alg\'ebrique symplectique propre sur $k$, de dimension $d=2m$. On suppose 
que toutes les $2$-formes sur $X$ sont ferm\'ees.
Alors l'invariant de Hasse-Witt de $X$ est non nul
si et seulement si l'application 
$\displaystyle{\H^0(X,{\Fabs}_*Z\Omega_X^2) \lra \H^0(X,\Omega_X^2)}$
induite par l'op\'erateur de Cartier est surjective.
\end{lemm}

\subsection{Scindage de Frobenius} 
Soit $X$ une vari\'et\'e alg\'ebrique sur $k$.
On note encore $\Fabs : X \to X$ le morphisme de Frobenius absolu. On dit, suivant Mehta et Ramanathan
(\cite{metha_ramanathan85}), que
l'application de Frobenius est scind\'ee (on dit que $X$ est \og Frobenius split\fg~en anglais) si l'homomorphisme de
$\cO_X$-modules $\Fabs^\sharp : \cO_X \to {\Fabs}_*\cO_X$
l'est, \cad, s'il existe un morphisme $\cO_X$-lineaire $\varphi : {\Fabs}_*\cO_X \to \cO_X$ tel que
$\varphi\circ\Fabs^\sharp$ soit l'application identit\'e de $\cO_X$.

\medskip

On termine ce paragraphe avec un \'enonc\'e dont on ne se servira pas par la suite faisant le lien entre
l'existence d'un scindage de Frobenius et l'invariant de Hasse-Witt de $X$. 

\begin{lemm}\label{lemme:scindage}
Soit  $X$ une vari\'et\'e symplectique irr\'eductible propre sur $k$
de dimension $d=2m$.
Soit $\Omega\in\textup{H}^0(X,\Omega_X^2)\setminus \{0\}$.
On a alors $C(\Omega)\neq 0$ si et seulement si l'application de Frobenius est
scind\'ee.
\end{lemm}
\begin{proof}
On sait que l'application
$\cO_X$-lin\'eaire 
$$\cO_X \lra {\Fabs}_*\cO_X$$
est scind\'ee si et seulement si \og l'application trace\fg~
$${\Fabs}_*Z\Omega_X^{d}={\Fabs}_*\omega_X \lra \omega_X$$ 
l'est (voir par exemple \cite[Proposition 1.3.7]{brion_kumar})~; c'est encore \'equivalent au fait que 
l'application
$$\H^0(X,{\Fabs}_*Z\Omega_X^m)=\H^0(X,{\Fabs}_*\Omega_X^m) \lra \H^0(X,\Omega_X^m)$$
soit non nulle puisque $\Omega^{\wedge m}$ est un g\'en\'erateur de $\omega_X$ en tout point et $k$ est un corps
parfait. On termine comme dans la d\'emonstration du lemme \ref{lemm:hassewittversusscindage}.
\end{proof}

\section{Sur un corps fini} 

Soit $k$ un corps fini 
de caract\'eristique $p$
\`a $q=p^a$ \'el\'ements, o\`u $a$ est un entier naturel non nul.  On note $W=W(k)$ l'anneau des vecteurs de Witt de
$k$,
$K_0=W[\frac{1}{p}]$ son corps des fractions et $\sigma$ l'automorphisme de Frobenius de $W$.
Soit  $X$ une vari\'et\'e sur $k$.
Le morphisme de Frobenius absolu $\Fabs : X \to X $ induit par transport de structures un endomorphisme
$\sigma$-lin\'eaire de la cohomologie cristalline $\Hcris^{\bullet}(X/W)$ not\'e $\Fcris$.
\begin{lemm}\label{lemm:scindageversusfrobenius}
Soit  $X$ une vari\'et\'e alg\'ebrique projective et lisse sur $k$ de dimension $d$.
On suppose 
$\Hcris^{\bullet}(X/W)$ sans torsion et on suppose que la suite spectrale de Hodge 
$E_1^{ij}=\H^j(X,\Omega_{X}^i)\Longrightarrow \Hdr^{i+j}(X/k)$
d\'eg\'en\`ere en $E_1$.
Soit $s \in \{0,\ldots,d\}$.
Si l'une des valeurs propres de l'endomorphisme
$K_0$-lin\'eaire $\Fcris^a$ de $\Hcris^s(X/W)\otimes_W K_0$ a une valuation $p$-adique nulle (o\`u la valuation est
normalis\'ee par la condition que la valuation $p$-adique de $q=p^a$ est 1) alors 
il existe $\Omega\in \H^0(X,Z\Omega_X^s)$ telle que
$C(\Omega)\neq 0$.
\end{lemm}
\begin{proof}
Si $(M,F)$ est un $F$-isocristal sur $k$ et $\lambda\in\bQ$, on note $M_\lambda$ le plus grand 
sous-F-isocristal de $M$ de pente $\lambda$ et, pour $I\subset\bQ$, on note 
$M_I=\bigoplus_{i\in I} M_{\lambda}$ (voir par exemple \cite[II.3.4]{illusie79}).

On note $\bar k$ une cl\^oture alg\'ebrique de $k$, $W(\bar k)$ l'anneau des vecteurs de Witt de $\bar k$,
$K_0(\bar k)=W(\bar k)[\frac{1}{p}]$ son corps des fractions et $\bar \sigma$ l'automorphisme de Frobenius de $W(\bar
k)$. On note enfin $\Fcrisbar$ l'endomorphisme de $\Hcris^{\bullet}(X\otimes\bar k/W(\bar k))$ induit par le Frobenius
absolu de $X\otimes\bar k$.
On a alors un isomorphisme canonique $\Hcris^{\bullet}(X\otimes\bar k/W(\bar k))\simeq \Hcris^{\bullet}(X/W)\otimes_{W}
W(\bar k)$ tel que $\Fcrisbar=\Fcris\otimes\bar\sigma$
car $\Hcris^{\bullet}(X/W)$ est sans torsion. On rappelle que pour tout entier 
$i\ge 0$, les pentes du $F$-isocristal 
$(\Hcris^{i}(X/W),\Fcris)$ sur $k$ sont, par d\'efinition, les pentes du $F$-isocristal
$(\Hcris^{i}(X\otimes\bar k/W(\bar k)),\Fcris)$ sur $\bar k$.

On consid\`ere le $F$-isocristal $(M=\Hcris^s(X/W)\otimes_W K_0,F=\Fcris\otimes \Id_{K_0})$.
On rappelle qu'alors les pentes du $F$-isocristal $M$ sont les valuations $p$-adiques des valeurs propres
de l'endomorphisme $\Fcris^a\otimes \Id_{K_0}$ de $\Hcris^s(X/W)\otimes_W K_0$ (voir \cite{manin63}). On a donc
$$\dim(M_{[0]})\ge 1$$ 
par hypoth\`ese.

On note $\eta$ la classe d'une section hyperplane de $X$ dans $\Hcris^2(X/W)\otimes_W K_0$.
On d\'eduit du 
th\'eor\`eme de Lefschetz fort (\cite{katz-messing74}) et du th\'eor\`eme de dualit\'e de Poincar\'e
(\cite{berthelotLN74}) que
l'application
\begin{eqnarray*}
M \otimes M & \lra &  \Hcris^{2d}(X/ W)\otimes_W  K_0(s-d) \simeq K_0(s)\\
\alpha\otimes\beta & \mapsto & \alpha\cdot\beta\cdot\eta^{d-s}
\end{eqnarray*}
induit un isomorphisme (d\'ependant du choix de $\eta$)
$$\textup{Hom}_{K_0}(M,K_0)\simeq M(-s)$$
compatible \`a l'action de Frobenius. On a donc
$$\dim(M_{[s]})=\dim(M_{[0]})\ge 1.$$
On en d\'eduit aussi que les pentes de $M$ sont dans $[0,s]$ (voir \'egalement \cite{berthelot73}).

On rappelle qu'il existe une \og $W$-alg\`ebre diff\'erentielle gradu\'ee strictement
anti-commutative\fg\\
$W\Omega_X^{\bullet}$ (d\'ependant fonctoriellement de $X$) appel\'ee complexe de de Rham-Witt
dont la cohomologie calcule
la cohomologie cristalline $\Hcris^{\bullet}(X/W)$ (\cite{illusie79}). On en d\'eduit l'existence d'une suite spectrale,
dite suite spectrale des pentes,
$$E_1^{ij}=\H^j(X,W\Omega_X^i)\Longrightarrow \Hcris^{i+j}(X/W)$$
o\`u, pour tout $i\ge 0$, les $K_0$-espaces vectoriels $\H^j(X,W\Omega_X^i)\otimes_{W}K_0$ sont de dimensions finies
(voir \cite[Th\'eor\`eme II.2.13]{illusie79}) et $\H^0(X,W\Omega_X^i)$ est un $W$-module libre de type fini 
(voir \cite[Corollaire II.2.17]{illusie79}). Le Frobenius absolu de $X$ induit un
endomorphisme (de $W$-alg\`ebre diff\'erentielle gradu\'ee) de
$W\Omega_X^{\bullet}$ not\'e \'egalement $\Fcris$.
La suite spectrale des pentes d\'eg\'en\`ere en $E_1$ et
fournit des 
isomorphismes de $F$-isocristaux (voir \cite[Corollaire II.3.5]{illusie79})
$$(\H^j(X,W\Omega_X^i)\otimes_{W}K_0,\Fcris\otimes\Id_{K_0})=({\Hcris^{i+j}(X/W)\otimes_{W}K_0}_{[i,i+1[},
\Fcris\otimes\Id_{K_0})$$
et on a donc 
$$(\H^0(X,W\Omega_X^s)\otimes_{W}K_0,\Fcris\otimes\Id_{K_0})=(({\Hcris^s(X/W)\otimes_{W}K_0})_{[s]},
\Fcris\otimes\Id_{K_0}).$$

On note $\alpha_1,\ldots,\alpha_r$ les valeurs propres
de l'endomorphisme $\Fcris^a\otimes \Id_{K_0}$ de $\Hcris^s(X/W)\otimes_W K_0$.
On rappelle que 
le
polygone de Newton 
du $F$-isocristal $M=\Hcris^s(X/W)\otimes_W K_0$
est le polygone de $\bR^2$ obtenu en portant bout \`a bout 
\`a partir du point $(0,0)$
des segments de pente la valuation
$p$-adique de $\alpha_i$ et de projection horizontale la multiplicit\'e de la valeur propre $\alpha_i$.
Le
polygone de Hodge \og g\'eom\'etrique\fg~de $X$ sur $k$ est le polygone de $\bR^2$ obtenu en portant bout \`a bout des
segments de pente $i$ et de projection horizontale le nombre de Hodge $h^{i,s-i}:=\dim_k(\H^{s-i}(X,\Omega_X^i))$
partant \'egalement du point $(0,0)$.

On sait, d'apr\`es le th\'eor\`eme de Mazur et Ogus (\cite{mazur72}, \cite{mazur73}et \cite{berthelot_ogus78}), que
le polygone de
Newton de
$(\Hcris^2(X/W)\otimes_W K_0,\Fcris\otimes\Id_{K_0})$ est au-dessus du polygone de Hodge \og g\'eom\'etrique\fg~de $X$
sur $k$ et qu'ils ont
m\^emes extr\'emit\'es. On d\'eduit alors de \cite[Theorem 1.6.1]{katz79} que le $F$-cristal 
($\H^0(X,W\Omega_X^s)$, $\Fcris$) est un facteur direct de ($\Hcris^{s}(X/W),\Fcris$).

On sait \'egalement que l'endomorphisme (de $W$-alg\`ebre diff\'erentielle gradu\'ee) de $W\Omega_X^{\bullet}$ induit
par le Frobenius $\Fabs$ est \og divisible\fg~par $p^i$ sur $W\Omega_X^i$ ($i\ge 0$) d'apr\`es \cite[I.2.19]{illusie79}.
On en d\'eduit que le $F$-isocristal ($\H^0(X,W\Omega_X^s)\otimes_{W}K_0$, $\Fcris\otimes\Id_{K_0}$) est pur de pente
$s$
puis que
le $F$-isocristal ($\H^0(X,W\Omega_X^s)\otimes_{W}K_0$, $\frac{1}{p^s}\Fcris\otimes\Id_{K_0}$) est pur de pente $0$.

Le $F$-cristal (sur $\bar k$) 
$(\H^0(X\otimes\bar k, W\Omega_{X\otimes\bar k}^s)=
\H^0(X, W\Omega_{X}^s)\otimes_{W} W(\bar k),\frac{1}{p^s}\Fcris \otimes \bar\sigma)$
est donc certainement un cristal unit\'e
(voir \cite[2.1]{katz72}) et poss\`ede une $W(\bar k)$-base $(W\bar \Omega_t)_{t\in T}$ telle que
$\frac{1}{p^s}\Fcris(W\bar \Omega_t)=W\bar \Omega_t$ pour tout $t\in T$ o\`u $W\bar \Omega_t\in \H^0(X\otimes\bar k,
W\Omega_{X\otimes\bar k}^s)$. On note au passage que $T$ est non vide puisque $\dim(M_{[s]})\ge 1.$

La projection $W\Omega_{X\otimes\bar k}^{\bullet} \to \Omega_{X\otimes\bar k}^{\bullet}$ induit un morphisme de la suite
spectrale des
pentes 
$$E_1^{ij}=\H^j(X\otimes\bar k,W\Omega_{X\otimes\bar k}^i)\Longrightarrow \Hcris^{i+j}(X\otimes\bar k/W(\bar k))$$ 
vers la suite spectrale de Hodge
$$E_1^{ij}=\H^j(X\otimes\bar k,\Omega_{X\otimes\bar k}^i)\Longrightarrow \Hdr^{i+j}(X\otimes\bar k/\bar k).$$
On sait, par hypoth\`ese, que la suite spectrale de Hodge d\'eg\'en\`ere en $E_1$~; en particulier, toutes les formes
diff\'erentielles sur $X$ ou $X\otimes\bar k$ sont ferm\'ees.
On a donc un diagramme commutatif

$$
\centerline{
\xymatrix{
\H^0(X\otimes\bar k,W\Omega_{X\otimes\bar k}^s)\otimes \bar k \ar[r] \ar[d] &  \Hcris^{s}(X\otimes\bar k/W(\bar
k))\otimes \bar k \ar[d] \\
\H^0(X\otimes\bar k,  Z\Omega_{X\otimes\bar k}^s)  \ar[r] & \Hdr^{s}(X\otimes\bar k/\bar k) 
}
}
$$
o\`u, la fl\`eche verticale de droite, d\'eduite de la suite exacte des coefficients universels (voir
\cite[VII 1.1.11]{berthelotLN74}, est un isomorphisme. On sait aussi que la
fl\`eche horizontale du haut est injective puisque 
$\H^0(X\otimes\bar k,W\Omega_{X\otimes\bar k}^s)$ est un facteur direct de 
$\Hcris^{s}(X\otimes\bar k/W(\bar k))$. On obtient ainsi une fl\`eche injective
$$\H^0(X\otimes\bar k,W\Omega_{X\otimes\bar k}^s)\otimes \bar k \hookrightarrow 
\H^0(X\otimes\bar k,  Z\Omega_{X\otimes\bar k}^s).$$
On en d\'eduit 
que, pour tout $t\in T$, l'image $\bar\Omega_t$ de $W\bar\Omega_t$
par
$$\H^0(X\otimes\bar k,W\Omega_{X\otimes\bar k}^s) \to \H^0(X\otimes\bar k,W\Omega_{X\otimes\bar k}^s)\otimes \bar k \to
\H^0(X\otimes\bar k, 
Z\Omega_{X\otimes\bar k}^s)=\H^0(X\otimes\bar k, \Omega_{X\otimes\bar k}^s)$$
est non nulle.

On rappelle enfin que, par choix de $W\bar\Omega_t$, on a $\frac{1}{p^s}\Fcris(W\bar\Omega_t)=W\bar\Omega_t$. Or
$\frac{1}{p^s}\Fcris$
rel\`eve l'op\'erateur $C^{-1}$, autrement dit le diagramme
$$
\centerline{
\xymatrix{
\H^0(X\otimes\bar k,W\Omega_{X\otimes\bar k}^s) \ar[r]^{\frac{1}{p^{s}}\Fcris} \ar[d] &  \H^0(X\otimes\bar
k,W\Omega_{X\otimes\bar k}^s) \ar[d] \\
\H^0(X\otimes\bar k, \Omega_{X\otimes\bar k}^s)  \ar[r] \ar[d]_{C^{-1}} & 
\H^0(X\otimes\bar k, \Omega_{X\otimes\bar k}^s/B\Omega_{X\otimes\bar k}^s) \\
\H^0(X\otimes\bar k,\cH^s(\Omega_X^{\bullet})) \ar@{=}[r]
 & \H^0(X\otimes\bar k, Z\Omega_{X\otimes\bar k}^s/B\Omega_{X\otimes\bar k}^s) \ar[u]
}
}
$$
est commutatif (voir \cite[Proposition 3.3]{illusie79}). On a donc $C(\bar \Omega_t)\neq 0$ pour tout $t\in T$. On en
d\'eduit facilement qu'il existe $\Omega\in \H^0(X,Z\Omega_X^s)$ telle que
$C(\Omega)\neq 0$.
\end{proof}

\section{La m\'ethode d'Ogus, Bogomolov et Zahrin}

On consid\`ere un corps de
nombres $K$ et on fixe une cl\^oture alg\'ebrique $\bar K$ de $K$. Soit $\ell$ un nombre premier.

\medskip

On rappelle dans un premier temps l'\'enonc\'e de la conjecture de Tate (\cite{tate65} et
\cite{tate94}) sur les corps de nombres.

\begin{conj}[Tate]Soit $X$
une vari\'et\'e projective et lisse sur $K$, g\'eom\'etriquement int\`egre.
L'application cycle induit une surjection
$$\NS(X\otimes \bar K))\otimes_{\bZ}\bQ_\ell
\twoheadrightarrow\bigcup_U
\Het^2(X\otimes \bar K,\bQ_\ell)(1)^U$$
o\`u $U$ parcourt l'ensemble des sous-groupes ouverts de $\Gal(\bar K/K)$.
\end{conj}

Soit $v$ une place finie de $K$. Le corps r\'esiduel $k_v$
est un corps fini \`a $N_v=p_v^{a_v}$ \'el\'ements o\`u $p_v$ est la caract\'eristique du corps $k_v$ et 
$a_v$ est le degr\'e de l'extension $k_v$ sur $\bF_p$.
Si $V$ est un $\bQ_\ell$-espace vectoriel de dimension finie, 
$v$ est une place de $K$ et
$\rho : \Gal(\bar K/K) \lra \Gl(V))$
est une rep\'esentation non ramifi\'ee en $v$, on note $F_{v,\rho}\in \Gl(V)$ un \og \'el\'ement de
Frobenius\fg~(ou plus exactement la
classe de conjugaison correspondante). 

\medskip

On peut maintenant rappeler l'\'enonc\'e (d'un cas particulier) de la conjecture
de semi-simplicit\'e (voir \cite{katz91} et \cite{tate94}) sur les corps finis.

\begin{conj}[\og Conjecture de semi-simplicit\'e\fg]Soit $X$ une vari\'et\'e propre et lisse sur $K$. 
On consid\`ere la repr\'esentation 
$\rho : \Gal(\bar K/K) \lra \Gl(\Het^2(X\otimes \bar K,\bQ_\ell))$.
Il existe alors un ensemble fini $\cV$ de places ultram\'etriques de $K$ tels que, si $v$ est une place finie de
$K$ et $v\not\in \cV$ alors $\rho$ est non ramifi\'ee en $v$ et 
$F_{v,\rho}$ est semi-simple.
\end{conj}

On aura besoin dans la suite du r\'esultat facile suivant.

\begin{lemm}\label{lemm:representationnontriviale}
Soit $X$ une vari\'et\'e projective et lisse sur $K$, g\'eom\'etriquement int\`egre. On suppose que la
conjecture de Tate est vraie pour $X$.
On suppose enfin 
$$\rang(\NS(X\otimes \bar K)) < \dim_{\bQ_\ell}(\Het^2(X\otimes \bar K,\bQ_\ell)).$$
On note $G$ l'image du groupe de Galois absolu $\Gal(\bar K/K)$ par la repr\'esentation $\ell$-adique
$\Gal(\bar K/K) \lra \Gl(\Het^2(X\otimes \bar K,\bQ_\ell)(1))$. 
Alors le groupe de Lie
$\ell$-adique $G$
est de
dimension $\ge 1$.
\end{lemm}

\begin{proof}
On suppose par l'absurde que $G$ est un groupe fini. Il existe donc une extension galoisienne finie $K\subset L\subset
\bar K$ telle que le groupe de Galois absolu
$U:=\Gal(\bar K/L)$ agisse trivialement sur $\Het^2(X\otimes \bar K,\bQ_\ell)(1)$. 
On a alors
$$\Het^2(X\otimes \bar K,\bQ_\ell)(1)^U=\Het^2(X\otimes \bar K,\bQ_\ell)(1)$$
et l'application cycle induit donc une surjection
$$\NS(X\otimes \bar K))\otimes_{\bZ}\bQ_\ell
\twoheadrightarrow
\Het^2(X\otimes \bar K,\bQ_\ell)(1).$$
On a obtenu la contradiction cherch\'ee.
\end{proof}

On note $\cO_K$ l'anneau des entiers de $K$. On consid\`ere une vari\'et\'e
$X$ propre et lisse sur $K$.
Soit $f : \cX \to \Spec(\cO_K)$ un mod\`ele entier de $X$, autrement dit, $\cX$ est un sch\'ema int\`egre et
$\cX\otimes K$ est isomorphe \`a $X$ sur $K$.
Soit $\cV$ un ensemble fini de places non archim\'ediennes de $K$ 
tel que $f$ soit lisse au-dessus de $\Spec(\cO_K[\cV^{-1}])$.
On consid\`ere la repr\'esentation 
$\rho : \Gal(\bar K/K) \lra \Gl(\Het^2(X\otimes \bar K,\bQ_\ell))$.
On fixe une place finie $v$ de $K$, $v\not\in\cV$, et on note $\Fabsv$ le Frobenius absolu de $\cX_v$.
On rappelle que, par le th\'eor\`eme de changement de base lisse (\cite[Expos\'e XVI]{sga4}) et \cite[Expos\'e
XV]{sga5},
les valeurs propres du \og Frobenius g\'eom\'etrique\fg~$F_{v,\rho}^{-1}$ 
sont aussi les valeurs propres de 
$\Fabsv\times \Id_{\bar{k}_v}$
agissant sur 
$\Het^2(\cX_v\otimes\bar k_v,\bQ_\ell)$. Le polyn\^ome 
$\det(\Id_{\Het^2(\cX_v\otimes\bar k_v,\bQ_\ell)} - F_{v,\rho}^{-1}t)$ 
est \`a coefficients entiers (ind\'ependants de $\ell\neq p_v$) d'apr\`es \cite{deligne:weil1}~: les
valeurs
propres de $F_{v,\rho}^{-1}$ sont donc des entiers alg\'ebriques.

\medskip

Le r\'esultat suivant est essentiellement d\^u \`a Ogus (\cite[Expos\'e VI]{deligneLN900}) d'une part,
Bogomolov et Zahrin (\cite{bogomolovzahrin}) d'autre part.

\begin{lemm}\label{lemm:bogomolov}
Soit $X$ une vari\'et\'e projective et lisse sur $K$, g\'eom\'etriquement int\`egre. On suppose que la
conjecture de Tate et la
conjecture de semi-simplicit\'e sont vraies pour $X$. 
On suppose \'egalement $\ell > 2 b_2$ o\`u $b_2:=\dim_\bQ(\Hb(X\otimes\bC,\bQ)$ et
on note $\rho$ la repr\'esentation
$$\Gal(\bar K/K) \lra \Gl(\Het^2(X\otimes \bar K,\bQ_\ell)).$$
On suppose enfin 
$$\rang(\NS(X\otimes \bar K)) < \dim_{\bQ_\ell}(\Het^2(X\otimes \bar K,\bQ_\ell)).$$
Il existe alors un ensemble $\Sigma$ de places ultram\'etriques de $K$ de densit\'e $>0$ 
tel que, si $v\in\Sigma$ alors 
\begin{enumerate}
\item[$\bullet$] $k_v$ est un corps premier de caract\'eristique $p_v\neq\ell$,
\item[$\bullet$] $\rho$ est non ramif\'ee en $v$,
\item[$\bullet$] les valeurs propres du
\og Frobenius g\'eom\'etrique\fg~$F_{v,\rho}^{-1}$ sont des entiers alg\'ebriques et l'une d'elles
a une valuation $p_v$-adique
nulle (o\`u la valuation est normalis\'ee par la condition que la valuation $p_v$-adique de $p_v$ est 1). 
\end{enumerate}

On peut \'egalement supposer la densit\'e de $\Sigma$ \'egale \`a 1 quitte \`a remplacer $K$ par une extension finie $L$
de $K$
convenable.
\end{lemm}

\begin{proof}
On note $\chi_{\ell} : \Gal(\bar K/K) \to \bZ_\ell^*$ le caract\`ere cyclotomique $\ell$-adique~; l'action
de $\Gal(\bar K/K)$ sur le module de Tate $\ell$-adique $\bZ_\ell(1)$ est induite par $\chi_\ell$. 
On rappelle que si $v$
est une place de $K$ qui n'est pas divisible par $\ell$ alors $\chi_\ell$
est non ramifi\'e en $v$ et $\chi_\ell(F_v)=p_v$.
On consid\`ere les repr\'esentations
$$\rho : \Gal(\bar K/K) \lra \Gl(\Het^2(X\otimes \bar K,\bZ_\ell))
\subset \Gl(\Het^2(X\otimes \bar K,\bQ_\ell))$$
et 
$$\rho\otimes\chi_\ell : \Gal(\bar K/K) \lra \Gl(\Het^2(X\otimes \bar K,\bZ_\ell)(1))
\subset \Gl(\Het^2(X\otimes \bar K,\bQ_\ell)(1)).$$

On consid\`ere un ensemble $\cV$ de places ultram\'etriques de $K$ tel que, si $v$ est une place finie de
$K$ et $v\not\in \cV$ alors $\rho$ est non ramifi\'ee en $v$ et 
$F_{v,\rho}$ est semi-simple.

On peut toujours supposer, quitte \`a \'elargir $\cV$, que si $v$ est une place finie de
$K$ et $v\not\in\cV$ alors $p_v\neq\ell$. On peut \'egalement supposer, d'apr\`es la discussion ci-dessus, que
les valeurs propres de $F_{v,\rho}^{-1}$
sont des entiers alg\'ebriques.

On consid\`ere une extension galoisienne finie $L$ de $K$ contenant toutes les racines $\ell$-i\`eme de l'unit\'e
telle que la
repr\'esentation induite par $\rho$
$$ \Gal(\bar K/L) \lra \Gl(\Het^2(X\otimes \bar K,\bZ_\ell)\otimes (\bZ_\ell/\ell \bZ_\ell))$$
soit triviale.

On sait d'apr\`es le th\'eor\`eme de densit\'e de \v{C}ebotarev que l'ensemble $\Sigma_1$ des places finies de $K$
de corps r\'esiduels premiers et
totalement d\'ecompos\'ees dans $L$ est de densit\'e $>0$.
On sait aussi, d'apr\'es \cite[Th\'eor\`eme 10]{serre81} et le lemme \ref{lemm:representationnontriviale}, que
l'ensemble $\Sigma_2$ des places finies $v$ de $K$ hors de $\cV$ (la repr\'esentation $\rho\otimes\chi_\ell$ est non
ramifi\'ee
en dehors de $\cV$ puisque $\rho$ l'est) telles que 
$F_{v,\rho\otimes\chi_\ell}\neq \Id_{\Het^2(X\otimes \bar K,\bQ_\ell)(1)}$ ou encore 
$F_{v,\rho}\neq p_v^{-1}\Id_{\Het^2(X\otimes \bar K,\bQ_\ell)}$ a une densit\'e \'egale \`a 1.

On pose $\Sigma:=\Sigma_1\cap\Sigma_2$~; c'est un ensemble de places ultram\'etriques de $K$ (en lesquelles $\rho$ est
non ramifi\'ee) de densit\'e $>0$ et disjoint de $\cV$.

Il reste, pour terminer la d\'emonstration du lemme, \`a voir que si $v\in\Sigma$ alors l'une des valeurs
propres de
$F_{v,\rho}$ a une valuation $p_v$-adique
nulle. 

Soit $v\in\Sigma$. On suppose, par l'absurde, que toutes les valeurs propres du 
\og Frobenius g\'eom\'etrique\fg~$F_{v,\rho}^{-1}$ ont une valuation $p_v$-adique
(enti\`ere puisque $k_v$ est un corps premier) $>0$.
On a donc $p_v\,|\,\Tr(F_{v,\rho}^{-1})$ (dans $\bZ$).
On en d\'eduit que $\frac{1}{p_v}F_{v,\rho}^{-1}$ est unipotent
d'apr\`es \cite[VI. Proposition 2.7]{deligneLN900}. On va donner l'argument parce qu'il explique l'hypoth\`ese $\ell> 2
b_2$ et le r\^ole de l'extension $L$. On
consid\`ere $w$ une place finie de $L$ au-dessus de $v$. On note $t_w$ la trace de
$F_{w,\rho}^{-1}$. 
On note que, puisque $a_w=1$, $F_{w,\rho}=F_{v,\rho}$. On a $t_w=p_w t'_w$ avec $t'_w\in\bZ$ par
hypoth\`ese.
On a bien s\^ur $t_w=\sum_{1\le i\le b_2}\alpha_i$ o\`u les $(\alpha_i)_{1\le i\le b_2}$ sont les valeurs propres
de $F_{w,\rho}^{-1}$. On rappelle que $\alpha_i$ est un entier alg\'ebrique de module $p_w^{a_w}=p_v$
(\cite{deligne:weil1}) et que $t_w\in\bZ$ (\textit{loc. cit}).
On a alors d'une part 
$t_w\equiv b_2\text{ modulo }\ell$ puisque la repr\'esentation $\Gal(\bar K/L) \lra \Gl(\Het^2(X\otimes \bar
K,\bZ_\ell)\otimes (\bZ_\ell/\ell \bZ_\ell))$ est triviale et $p_w-1\equiv 0\text{ modulo }\ell$ puisque $k_w$ est un
corps premier et $L$ contient un racine primitive $\ell$-i\`eme de l'unit\'e.
On a donc
$t'_w\equiv b_2\text{ modulo }\ell$ et $|t'_w|\le b_2$. On en d\'eduit que $t'_w=b_2$
puisque $\ell > 2 b_2$ (et $t'_w\in\bZ$) puis que $t_w=p_w b_2$ et finalement que $\alpha_i=\alpha_j$ pour tout 
$i, j$ dans $\{1,\ldots,b_2 \}$.
On a donc montrer que
$\frac{1}{p_w}F_{w,\rho}^{-1}=\frac{1}{p_v}F_{v,\rho}^{-1}$ est unipotent.
On en d\'eduit que
$F_{v,\rho}=\frac{1}{p_v}\Id_{\Het^2(X\otimes \bar K,\bQ_\ell)}$ puisque, par hypoth\`ese,
$F_{v,\rho}$ est semi-simple. On a
ainsi obtenu la
contradiction souhait\'ee.

On obtient la deuxi\`eme assertion du lemme en rempla\c cant $K$ par $L$~; en effet, d'apr\`es le
th\'eor\`eme de densit\'e de \v{C}ebotarev, l'ensemble des places finies de $L$ dont le corps r\'esiduel est premier
a une densit\'e 1.
\end{proof}

\section{La construction de Kuga, Satake et Deligne}

On \'etend dans ce paragraphe un r\'esultat de Deligne (\cite{deligne72}) utilisant une construction due \`a
Kuga et Satake (\cite{kugasat67})~;
Deligne explique dans \textit{loc. cit.} comment la structure de Hodge $\Hb^2(X_0,\bC)$ d'une surface $K 3$ $X_{0}$
s'exprime \`a l'aide des vari\'et\'es ab\'eliennes. On va expliquer que c'est encore le cas lorsque $X_{0}$ est une
vari\'et\'e symplectique irr\'eductible. Le r\'esultat avait d\'ej\`a \'et\'e observ\'e (de fa\c con ind\'ependante) par
d'autres auteurs (voir par exemple \cite{andre96}).

\medskip

On commence par expliquer ce qui doit \^etre modifi\'e dans \cite{deligne72}.

\medskip

Soit $(X_{0},\eta_{0})$ une vari\'et\'e symplectique irr\'eductible polaris\'ee sur $\bC$
($\eta_{0}\in\Pic(X_{0})$ est la classe d'un faisceau inversible $L_{0}$ ample sur $X_{0}$).
On pose $b_2=b_2(X_0)$ et $d=\dim(X_0)$.

On note $(\hat f : \hat X \to\hat S,\hat L)$ une d\'eformation universelle de $(X_{0},L_{0})$~; 
$\hat S$ est le spectre d'un anneau de series formelles de dimension $b_2-3$ (voir par exemple
\cite[1.15]{huybrechts99}).
On a suppos\'e $L_{0}$ ample sur $X_{0}$ et $\hat f$ est donc alg\'ebrisable. On a donc
un diagramme cart\'esien

$$
\centerline{
\xymatrix{
\hat X \ar[r] \ar[d]^{\hat f} &  X\ar[d]^{f} \\
\hat S  \ar[r] & S 
}
}
$$

\noindent o\`u $S$ est un sch\'ema de type fini sur $\bC$ et $X$ est une famille de vari\'et\'es symplectiques
irr\'eductibles sur
$S$. On a aussi un faisceau inversible $L$ sur $X$, ample sur
$S$, induisant $\hat L$ sur $\hat X$. On peut m\^eme, gr\^ace au
th\'eor\`eme d'approximation d'Artin, trouver $S$ tel que $\hat S$ soit le compl\'et\'e
formel de $S$ en l'image $0$ du point ferm\'e $\hat 0$ de $\hat S$.

On note $\eta\in\Gamma(S,\Rb^2f_*\bZ)$ la classe de $L$. On d\'esigne par $\Pb^2f_*\bZ$ le noyau de
l'application
$\bZ$-lin\'eaire

\begin{eqnarray*}
\Rb^2f_*\bZ & \lra & \Rb^{2d}f_*\bZ \\
\alpha & \mapsto & \alpha\cdot\eta^{d-1}
\end{eqnarray*}

\noindent et on note 

$$\psi : \Pb^2f_*\bZ(1)\otimes \Pb^2f_*\bZ(1)\lra \underline{\bZ}$$

\noindent l'application induite par

\begin{eqnarray*}
\Rb^2f_*\bZ(1) \otimes \Rb^2f_*\bZ(1) & \lra &  \Rb^{2d}f_*\bZ(d) \simeq \underline{\bZ}\\
\alpha\otimes\beta   & \longmapsto & \alpha\cdot\beta\cdot\eta^{d-2}.
\end{eqnarray*}

On sait (voir \cite[Corollaire 1]{beauville83}) et c'est essentiel pour faire fonctionner l'argument de Deligne
que la variation de structures de Hodge polaris\'ees $(\Pb^2f_*\bZ(1),\psi)$ est une d\'eformation universelle de 
$(\Pb^2(X_{0},\bZ)(1),\psi_{0})$. On sait \'egalement que la forme
quadratique r\'eelle d\'eduite de $\psi_{0}$ est de signature $(2,b_2-3)$ (voir par exemple \cite{weil58}).

On fixe maintenant un nombre premier $\ell$. On peut bien s\^ur remplacer l'anneau des coefficients $\bZ$ par
$\bZ_\ell$ ci-dessus. On peut aussi consid\'erer la cohomologie $\ell$-adique. On obtient un faisceau lisse
$\Pet^2f_*\bZ_\ell(1)$ sur $S$ et une forme bilin\'eaire
$$\psi_\ell : \Pet^2f_*\bZ_\ell(1)\otimes \Pet^2f_*\bZ_\ell(1)\lra \underline{\bZ}_\ell,$$
qui, via l'isomorphisme de comparaison $\Pet^2f_*\bZ_\ell(1) \simeq \Pb^2f_*\bZ_\ell(1)$
(voir \cite[Expos\'e XI]{sga4}),
se d\'eduit de $\psi$ par extension des scalaires.

Soit $V$ un module libre sur un anneau $A$, muni d'une forme quadratique $Q$. On note
$C(V,Q)$ l'alg\`ebre de Clifford associ\'ee \`a $(V,Q)$ et $C^{+}(V,Q)$ sa partie paire.

On obtient le r\'esultat suivant en reprenant mot pour mot la d\'emonstration de 
\cite[Proposition 6.5]{deligne72}.

\begin{prop}\label{prop:deligne}
Soit $X_{0}$ une vari\'et\'e symplectique irr\'eductible polaris\'ee sur un corps $K\subset \bC$.
On pose $C:=C^{+}(\Pb^2(X_{0}\otimes\bC,\bZ)(1),\psi_{0})$.

\begin{enumerate}
\item Il existe~:
\begin{enumerate}
\item un sch\'ema $S$ de type fini sur $\bQ$, une famille $f : X \to S$ de vari\'et\'es symplectiques irr\'eductibles
polaris\'ees, param\'etr\'ee par $S$ et une extension finie $L$ de $K$ et un point $s$ de $S$ d\'efini sur $L$ tel que
$X_s$ soit isomorphe \`a $X_{0}\otimes L$~;
\item un sch\'ema ab\'elien $a : A \to S$ et $\mu : C \to \End_C(A)$~;
\item un isomorphisme de $\bZ_\ell$-faisceaux d'alg\`ebres sur $S$
$$u : C^{+}(\Pet^2f_*\bZ_\ell(1),\psi_\ell)\simeq\underline{\End}_C(\Ret^1a_*\bZ_\ell).$$
\end{enumerate}
\item On suppose $K\subset \bar \bQ$. Il existe alors un ensemble fini $\cV$ de places ultram\'etriques de
$L$
tel que $A_s$ ait potentiellement bonne
r\'eduction en toutes
les places finies de $L$ n'appartenant pas \`a l'ensemble $\cV$.
\end{enumerate}
\end{prop}

\begin{rema}\label{rema:deligne}
On peut pr\'eciser un peu la derni\`ere assertion de la proposition pr\'ec\'edente.
Soit $\cX \to \cS$ un mod\`ele entier de 
$f : X \to S$~; $\cS$ est un sch\'ema int\`egre de type fini sur $\bZ$, $S$ est la fibre g\'en\'erique de
$\cS\to\Spec(\bZ)$ et $X$ s'identifie \`a la fibre g\'en\'erique de $\cX \to \cS$ comme $S$-sch\'ema.
Il existe un ensemble fini $\cV$ de places ultram\'etriques de $L$ et un morphisme
$\Spec(\cO_L[\cV^{-1}]) \to \cS$ qui rel\`eve $\Spec(L) \to S \subset \cS$.
On peut aussi supposer $\cX\times_\cS \Spec(\cO_L[\cV^{-1}])$ lisse sur $\Spec(\cO_L[\cV^{-1}])$~; $\cV$
convient.
\end{rema}

\section{Sur un corps de nombres}

On consid\`ere un corps de
nombres $K$ et on fixe une cl\^oture alg\'ebrique $\bar K$ de $K$. Soit $\ell$ un nombre premier.

\medskip

Soit $X_{0}$ une vari\'et\'e symplectique irr\'eductible polarisable sur $K$.
On rappelle que le groupe de Galois $\Gal(\bar K/K)$ agit sur $\Het^2(X_0\otimes\bar K,\bQ_\ell)$ 
par transport de structures. On \'etudie dans ce paragraphe la repr\'esentation $\ell$-adique $\Gal(\bar K/K) \to
\Gl(\Het^2(X_0\otimes\bar K,\bQ_\ell))$.

\begin{theo}\label{theo:semisimplicite}
Soit $X_{0}$ une vari\'et\'e symplectique irr\'eductible polaris\'ee d\'efinie sur $K$. 
On suppose ou bien $\textup{rang}(\NS(X_0\otimes\bar K)) \ge 2$
ou bien $\dim_{\bQ_\ell}(\Het^2(X_0\otimes\bar K,\bQ_\ell))$ pair.
Il existe alors un ensemble fini $\cV$ de places ultram\'etriques de $K$ tels que, si $v$ est une place finie de
$K$ et $v\not\in\cV$ alors $\rho$ est non ramifi\'ee en $v$ et 
$F_{v,\rho}$ est semi-simple.
\end{theo}
\begin{proof}
Soit $\eta_0\in\Het^2(X_{0}\otimes \bar K,\bQ_\ell)(1)$ une polarisation de $X_{0}$ d\'efinie sur $K$. On a
une d\'ecomposition en somme directe
$$\Het^2(X_{0}\otimes \bar K,\bQ_\ell)(1)=\bQ_\ell\cdot\eta_0\oplus \Pet^2(X_{0}\otimes \bar K,\bQ_\ell)(1)$$
de $\Gal(\bar K/K)$-modules et il suffit donc de montrer l'assertion correspondante pour le $\Gal(\bar K/K)$-module
$\Pet^2(X_{0}\otimes \bar K,\bQ_\ell)(1)$. 

On consid\`ere $(S,X,f,L,A,a,\mu,u,\cV_L,s)$ comme
dans la proposition \ref{prop:deligne}. Soient $\cX \to \cS$ un
mod\`ele entier de 
$f : X \to S$.
On suppose comme dans la 
remarque \ref{rema:deligne}
qu'il existe un morphisme $\Spec(\cO_L[\cV_L^{-1}]) \to \cS$ relevant
$\Spec(L) \to S \subset \cS$ tel que
$\cX\times_\cS \Spec(\cO_L[\cV_L^{-1}])$ soit lisse sur $\Spec(\cO_L[\cV_L^{-1}])$.

On peut toujours supposer que $\cV_L$ contient toutes les places finies $v$ de
$L$ dont le corps r\'esiduel est de caract\'eristique $\ell$. 

On note $\cV$ l'ensemble des places finies de $K$ qui divisent l'une des places de $\cV_L$.
On note $\rho$ la repr\'esentation $\Gal(\bar K/K) \to \Gl(\Het^2(X_0\otimes\bar K,\bQ_\ell))$. On peut toujours
supposer, quitte \`a \'elargir $\cV$, que $\rho$ est non ramifi\'ee en $v\not\in\cV$.

On fixe une place finie $v$ de $K$ telle que $v\not\in\cV$ et $w$ une place finie de $L$ au-dessus de $v$, de sorte que
$w\not\in\cV_L$. Il suffit de montrer que $F_{w,\rho}$ est
semi-simple. On sait en effet qu'une puissance convenable de $F_{v,\rho}$ est un \'el\'ement de Frobenius en $w$ 
donc conjugu\'ee \`a $F_{w,\rho}$. On d\'eduit l'assertion cherch\'ee du lemme \ref{lemm:semisimplicite}.
On peut donc \'egalement supposer, quitte \`a remplacer $K$ par $L$, que $s$ est d\'efini sur $K$.

On sait, d'apr\`es le th\'eor\`eme de changement de base lisse 
(\cite[Expos\'e XVI]{sga4}), qu'on a un isomorphisme 
$$\Pet^2(X_s\otimes \bar K,\bQ_\ell)\simeq \Pet^2(\cX_v\otimes\bar{k}_v,\bQ_\ell)$$
par lequel l'action de $F_{v,\rho}^{-1}$ sur $\Pet^2(X_0\otimes \bar K,\bQ_\ell)$ 
correspond (\`a conjugaison pr\`es) \`a l'action de $\Fabs^{a_v}\times \Id_{\bar{k}_v}$ sur 
$\Pet^2(\cX_v\otimes\bar{k}_v,\bQ_\ell)$ (voir \cite[Expos\'e XV]{sga5}).
On sait aussi que, par choix de $v$, $A_s$ a potentiellement bonne r\'eduction en $v$.
Quitte \`a remplacer $K$ par une extension finie et $F_{v,\rho}$ par une puissance convenable de lui-m\^eme,
on peut donc supposer, en tenant compte \`a nouveau du
lemme \ref{lemm:semisimplicite}, que $A_s$ a bonne 
r\'eduction en $v$. On note $A_v$ sa r\'eduction en $v$.

On pose \`a nouveau $C:=C^{+}(\Pb^2(X_{0}\otimes\bC,\bZ)(1),\psi_{0})$.
On sait d'apr\`es la proposition \ref{prop:deligne} qu'on a
un isomorphisme de $\Gal(\bar K/K)$-modules
$$u\otimes {\bar K} : C^{+}(\Pet^2(X_s\otimes \bar K,\bZ_\ell)(1),\psi_\ell\otimes\bar K)\simeq 
\End_C(\Het^1(A_s\otimes \bar K,\bZ_\ell))$$
qui se r\'eduit \`a un isomorphisme de $\Gal(\bar k_v/k_v)$-modules
\begin{eqnarray}\label{isomorphisme:motif2}
u\otimes {\bar k_v} : C^{+}(\Pet^2(\cX_v\otimes\bar k_v,\bZ_\ell)(1),\psi_\ell\otimes\bar k_v)
\simeq \End_C(\Het^1(A_v\otimes\bar k_v,\bZ_\ell)).
\end{eqnarray}

\medskip

On suppose $\textup{rang}(\NS(X_0\otimes\bar K)) \ge 2$. On peut trouver
une classe alg\'ebrique
$\beta_0\in \Pet^2(X_{0}\otimes \bar K,\bQ_\ell)(1)$ lin\'eairement ind\'ependante de $\eta_0$ sur $\bQ_\ell$
telle que $(\psi_\ell\otimes\bar K)(\beta_0)\neq 0$. 
On peut toujours supposer $\beta_0\in \Pet^2(X_{0}\otimes \bar K,\bQ_\ell)(1)^{\Gal(\bar K/K)}$
quitte \`a remplacer $K$ par une extension finie.
On note $\beta_v\in\Pet^2(\cX_v\otimes\bar{k}_v,\bQ_\ell)(1)$ l'image de 
$\beta_0\in \Pet^2(X_s\otimes \bar K,\bQ_\ell)(1)$. On note au passage que $\beta_v^{\perp}$
(resp. $\beta_0^{\perp}$) est invariant sous $\Gal(\bar k_v/k_v)$ (resp. $\Gal(\bar K/K)$) puisque
$\bQ_\ell\cdot\beta_v^{\perp}$ (resp. $\bQ_\ell\cdot\beta_0^{\perp}$) l'est.

On rappelle que par d\'efinition, on a un isomorphisme de $\Gal(\bar k_v/k_v)$-modules
\begin{eqnarray}\label{isomorphisme:C}
C^{+}(\Pet^2(\cX_v\otimes\bar k_v,\bQ_\ell)(1),\psi_\ell\otimes\bar k_v)\simeq
\bigoplus_{i \ge 0}\bigwedge^{2i}(\Pet^2(\cX_v\otimes\bar k_v,\bQ_\ell)(1)).
\end{eqnarray}
On consid\`ere alors l'application
$\bQ_\ell$-lin\'eaire de
$\Gal(\bar k_v/k_v)$-modules
\begin{eqnarray*}
\Pet^2(\cX_v\otimes\bar k_v,\bQ_\ell)(1)\supset \beta_v^{\perp} & \hookrightarrow & \End_C(\Het^1(A_v\otimes\bar
k_v,\bQ_\ell)) \\
\nonumber \alpha & \mapsto & (u\otimes\bar k_v)(\alpha\wedge\beta_v)
\end{eqnarray*}
induite par (\ref{isomorphisme:motif2}).

On sait bien que l'action de $\Fabs^{a_v}\times \Id_{\bar{k}_v}$ sur 
$\Het^1(A_v\otimes\bar k_v,\bQ_\ell)$ est semi-simple
(voir par exemple \cite[VI. Lemma 2.10]{deligneLN900}). On en d\'eduit que 
$\Fabs^{a_v}\times \Id_{\bar{k}_v}$ agissant sur $\beta_v^{\perp}$ est semi-simple puis
que la restriction de $F_{v,\rho}$ \`a
$\beta_0^{\perp}$ l'est aussi. Or
$(\psi_\ell\otimes\bar K)(\beta_0)\neq 0$ et on a donc une d\'ecomposition en somme directe 
de $\Gal(\bar K/K)$-modules
$\Pet^2(X_s\otimes \bar K,\bQ_\ell)(1)=\bQ_\ell\cdot\beta_0\oplus\beta_0^{\perp}$. On conclut facilement dans ce cas.

\medskip

On suppose maintenant $\dim_{\bQ_\ell}(\Het^2(X_0\otimes\bar K,\bQ_\ell))=2r+2$
pour un entier $r\ge 1$. On a un isomorphisme de
$\Gal(\bar k_v/k_v)$-modules
\begin{eqnarray}\label{isomorphisme:wedge}
\bigwedge^{2r}\Pet^2(\cX_v\otimes\bar k_v,\bQ_\ell)
\simeq
\textup{Hom}_{\bQ_\ell}
(\Pet^2(\cX_v\otimes\bar k_v,\bQ_\ell),\bigwedge^{2r+1}\Pet^2(\cX_v\otimes\bar k_v,\bQ_\ell)).
\end{eqnarray}

On d\'eduit de (\ref{isomorphisme:motif2}), (\ref{isomorphisme:C}) et (\ref{isomorphisme:wedge}) 
que le $\Gal(\bar k_v/k_v)$-module
$\Pet^2(\cX_v\otimes\bar k_v,\bQ_\ell)(1)$ est un facteur direct du
$\Gal(\bar k_v/k_v)$-module
$\textup{Hom}_{\bQ_\ell}(\End(\Het^1(A_v\otimes\bar k_v,\bQ_\ell)),\bQ_\ell))$
puisqu'on a un isomorphisme de $\Gal(\bar k_v/k_v)$-modules
$\bigwedge^{2r+1}\Pet^2(\cX_v\otimes\bar k_v,\bQ_\ell)(2r+1)\simeq\bQ_\ell$.
On sait par ailleurs que
l'action de $\Fabs^{a_v}\times \Id_{\bar{k}_v}$ sur 
$\Het^1(A_v\otimes\bar k_v,\bQ_\ell)$ est semi-simple et celle de $F_{v,\rho}$ sur
$\Het^2(X_{s}\otimes \bar K,\bQ_\ell)$ l'est \'egalement.
Ceci termine la d\'emonstration du th\'eor\`eme.
\end{proof}

\begin{lemm}\label{lemm:semisimplicite}
Soit $V$ un espace vectoriel de dimension finie sur corps $K$ de caract\'eristique $0$ et soit $u$ un endomorphisme de
$V$. Soit $r$ un entier $\ge 1$. Si $u$ est inversible et $u^r$ est
semi-simple alors $u$ est \'egalement semi-simple.
\end{lemm}
\begin{proof} Soit $\bar K$ une cloture alg\'ebrique de $K$. On pose $\bar V:= V\otimes\bar K$ et
$\bar u:=u\otimes\Id_{\bar K}\in \End_{\bar K}(\bar V)$.
Quitte \`a remplacer $\bar V$ par l'un des sous-espaces propres de $\bar{u}^r$, on peut
toujours supposer qu'il existe 
$\lambda\in\bar K\setminus\{0\}$ tel que $\bar{u}^r=\lambda\,\Id_{\bar V}$. Soient
$\zeta\in\bar K$ une racine primitive $r$-i\`eme de l'unit\'e et
$P:=\prod_{i\in\{1,\ldots,r\}}(T-\zeta^i\lambda)\in\bar K[T]$~; $P$ est scind\'e, \`a racines simples et 
$P(\bar u) \equiv 0$.
On conclut facilement.
\end{proof}

\section{D\'emonstration du th\'eor\`eme \ref{theo:principal}}

On commence par un r\'esultat g\'en\'eral.

\begin{prop}\label{prop:hasswitt}
Soit  $X$ une vari\'et\'e alg\'ebrique projective et lisse sur $K$, g\'eom\'etriquement int\`egre, de dimension $d=2m$.
On suppose 
qu'il existe une $2$-forme $\Omega$ 
telle qu'on ait $\textup{H}^0(X,\Omega_X^2)=K\cdot\Omega$ et $\Omega^{\wedge m}\neq 0$.
On suppose enfin que la conjecture de Tate et
la conjecture de semi-simplicit\'e sont vraies pour $X$.
On consid\`ere un mod\`ele entier
$f : \cX \to \Spec(\cO_K)$ de $X$ et un ensemble fini
$\cV$ de places non archim\'ediennes de $K$ tel que $f$ soit lisse au-dessus de l'ouvert
$\Spec(\cO_K[\cV^{-1}])$ de $\Spec(\cO_K)$.
Il existe alors un ensemble $\Sigma$ de places ultram\'etriques de $K$ de densit\'e $>0$ 
tel que pour tout
$v\in\Sigma\setminus\cV$ l'invariant de Hasse-Witt de $\cX_v$ soit non nul. On peut supposer la densit\'e de
$\Sigma$  \'egale
\`a 1 quitte \`a remplacer $K$ par une extension finie $L$ de $K$
convenable.
\end{prop}

\begin{proof}
On peut toujours supposer, quitte \`a \'elargir $\cV$, que pour toute place finie $v\not\in\cV$
la cohomologie cristalline $\Hcris^{*}(\cX_v/W(k_v))$ est sans torsion sur $W(k_v)$ (\cite[Proposition 6.6.1]{joshi01})
et que
la suite spectrale de Hodge
$E_1^{ij}=\H^j(\cX_v,\Omega_{\cX_v}^i)\Longrightarrow \Hdr^{i+j}(\cX_v/k_v)$
d\'eg\'en\`ere en $E_1$ (\cite{deligneillusie87}).

On fixe un nombre premier $\ell > 2 \dim_\bQ(\Hb(X\otimes\bC,\bQ)$.
On note $\rho$ la repr\'esentation $\ell$-adique $\Gal(\bar K/K) \to \Gl(\Het^2(X\otimes\bar K,\bQ_\ell))$.
On sait d'apr\`es le lemme \ref{lemm:bogomolov}
qu'il
existe un ensemble $\Sigma$ de places ultram\'etriques de $K$ de densit\'e $>0$ 
tel que, si $v\in\Sigma$ alors 
\begin{enumerate}
\item[$\bullet$] $k_v$ est un corps premier de caract\'eristique $p_v\neq\ell$,
\item[$\bullet$] $\rho$ est non ramif\'ee en $v$,
\item[$\bullet$] les valeurs propres du
\og Frobenius g\'eom\'etrique\fg~$F_{v,\rho}^{-1}$ sont des entiers alg\'ebriques et l'une d'elles
a une valuation $p_v$-adique
nulle (o\`u la valuation est normalis\'ee par la condition que la valuation $p_v$-adique de $p_v$ est 1). 
\end{enumerate}

Soit $v\in \Sigma\setminus\cV$. On d\'eduit alors du lemme \ref{lemm:scindageversusfrobenius} qu'il existe 
$\Omega_v\in \H^0(\cX_v,Z\Omega_{\cX_v}^2)$ telle que
$C(\Omega_v)\neq 0$ et enfin du lemme \ref{lemm:hassewittversusscindage} que l'invariant de Hasse-Witt de 
$\cX_v$ est non nul.
\end{proof}

On est maintenant en mesure de d\'emontrer le th\'eor\`eme \ref{theo:principal}.

\begin{proof}[D\'emonstration du th\'eor\`eme \ref{theo:principal}]
On remarque que sous les hypoth\`eses du th\'eor\`eme, on a toujours $b_2(X)\ge 4$. On sait alors d'apr\`es
\cite[Theorem 1.6.1]{andre96} que la conjecture de Tate est vraie pour $X$. On d\'eduit maintenant le r\'esultat
cherch\'e de la proposition \ref{prop:hasswitt} et du th\'eor\`eme \ref{theo:semisimplicite}.

\end{proof}


\providecommand{\bysame}{\leavevmode ---\ }
\providecommand{\og}{``}
\providecommand{\fg}{''}
\providecommand{\smfandname}{et}
\providecommand{\smfedsname}{\'eds.}
\providecommand{\smfedname}{\'ed.}
\providecommand{\smfmastersthesisname}{M\'emoire}
\providecommand{\smfphdthesisname}{Th\`ese}

\end{document}